\documentclass[a4paper]{amsart}

\PII{}
\makeatletter
\def\@setcopyright{\@empty}
\makeatother

\newcommand{\E}{E_n(f)_{p,\alpha}}
\newcommand{\Epar}[2]{E_{#1}\left(#2\right)_{p,\alpha}}
\newcommand{\T}[2]{\tau_{#1}\left(#2\right)}
\newcommand{\hatT}[2]{\hat\tau_{#1}\left(#2\right)}
\newcommand{\sT}{T_{2;y}(f,x)}
\newcommand{\w}{\hat\omega(f,\delta)_{p,\alpha}}
\newcommand{\wpar}[1]{\hat\omega\left(#1\right)_{p,\alpha}}
\newcommand{\K}{K(f,\delta)_{p,\alpha}}
\newcommand{\Kpar}[1]{K\left(#1\right)_{p,\alpha}}
\newcommand{\norm}[1]{\left\|#1\right\|_{p,\alpha}}
\newcommand{\Si}[1]{(1-#1^2)}
\newcommand{\Co}[1]{\cos^4#1/2}
\newcommand{\Dx}{D_{x,2,2}}
\newcommand{\Du}{D_{u,0,4}}
\newcommand{\AD}{AD(p,\alpha)}
\newcommand{\Px}[1]{P_{#1}^{(2,2)}}
\newcommand{\Py}[1]{P_{#1}^{(0,4)}}
\newcommand{\fz}{\left(f(z)-\frac{c_1}{c_0}\right)}
\newcommand{\krn}[1]{%
  \left(
    \frac{\sin\frac{m#1}2}{\sin\frac{#1}2}
  \right)^{2q+4}}
\newcommand{\kap}{\kappa(\delta)}
\newcommand{\sincosv}{(\sin v/2)^{-1}(\cos v/2)^{-9}}
\newcommand{\sincosu}{\sin u/2(\cos u/2)^9}
\newcommand{\varsincosu}{%
  \left(\sin\frac u2\right)\left(\cos\frac u2\right)^5}
\newcommand{\Lp}{L_{p,\alpha}}
\newcommand{\Lmu}{L_{1,2}}
\newcommand{\gd}{g_{\delta}(x)}
\newcommand{\allp}{1\le p\le\infty}
\newcommand{\prn}[1]{\left(#1\right)}
\newcommand{\brc}[1]{\left\{#1\right\}}

\newtheorem{thm}{Theorem}[subsection]
\newtheorem{lmm}{Lemma}[subsection]
\newtheorem{cor}{Corollary}[subsection]

\newcounter{const}

\numberwithin{const}{thm}
\numberwithin{const}{lmm}
\numberwithin{const}{cor}

\makeatletter
\newcommand{\Cn}[1][]{%
  \stepcounter{const}C_{\theconst}%
  \@ifnotempty{#1}{\newcounter{#1}\setcounter{#1}{\arabic{const}}}}
\makeatother

\newcommand{\lastC}{C_{\theconst}}
\newcommand{\prevC}[1][1]{%
	{\countdef\n=255
	 \n=\theconst
	 \advance\n by-#1
	 C_{\number\n}}}

\numberwithin{equation}{subsection}

\renewcommand{\theconst}{\arabic{const}}

\begin{document}

\title[Theorems for a generalised modulus of smoothness]
	{Direct and inverse theorems of approximation theory for
	a generalised modulus of smoothness}
\author{M.~K.\ Potapov}
\address{M.~K.\ Potapov\\
	Department of Mechanics and Mathematics\\
	Moscow State University\\
	Moscow 117234\\
	Russia}
\author{F.~M.\ Berisha}
\address{F.~M.\ Berisha\\
	Faculty of Mathematics and Sciences\\
	University of Prishtina\\
	N\"ena Terez\"e~5\\
	10000 Prishtin\"e\\
	Kosov0}
	\email{faton.berisha@uni-pr.edu}

\date{}
\thanks{This work was done under the support
	of the Russian Foundation for Fundamental Scientific Research,
	Grant \#97-01-00010 and Grant \#96/97-15-96073.}

\begin{abstract}
	An asymmetric operator of generalised translation is introduced
	in this paper.
	Using this operator,
	we define a generalised modulus of smoothness
	and prove direct and inverse theorems of approximation theory
	for it.
\end{abstract}

\maketitle

\begin{center}
\textbf{Introduction}
\end{center}

In a number of papers
(see, e.g.,
\cite{butzer-s-w:c-80,pawelke:acta-72,potapov:vestnik-83,
	potapov-f:trudy-85%
}%
)
direct and inverse theorems of approximation theory
are proved for generalised moduli of smoothness
defined by means of symmetric operators of generalised translation.
It is of interest to obtain the same results
for a moduli of smoothness
defined by means of asymmetric operators of generalised translation.

In the present paper such an operator is introduced,
the generalised modulus of smoothness is defined by its means,
and direct and inverse theorems of approximation theory
are proved for that modulus.

\subsection{}

By~$L_p$ we denote the set of functions~$f$
such that in the case $1\le p<\infty$,
$f$~is measurable on the segment~$[-1,1]$
and
\begin{displaymath}
	\|f\|_p=\prn{\int_{-1}^1|f(x)|^p\,dx}^{1/p}<\infty;
\end{displaymath}
and in the case $p=\infty$,
the function~$f$ is continuous on the segment~$[-1,1]$,
and
\begin{displaymath}
	\|f\|_\infty=\max_{-1\le x\le1}|f(x)|.
\end{displaymath}

Denote by~$\Lp$ the set of functions~$f$
such that
$f(x)\*\Si{x}^\alpha\in L_p$,
and put
\begin{displaymath}
	\norm f=\|f(x)\Si{x}^\alpha\|_p.
\end{displaymath}

By~$\E$ we denote the best approximation of the function $f\in\Lp$
by algebraic polynomials of degree not greater than $n-1$,
in $\Lp$ metrics, i.e.,
\begin{displaymath}
	\E=\inf_{P_n\in\mathbb P_n}\norm{f-P_n},
\end{displaymath}
where~$\mathbb P_n$ is the set of algebraic polynomials
of degree not greater than $n-1$.

By~$D_{x,\nu,\mu}$ we denote the operator
\begin{displaymath}
	D_{x,\nu,\mu}
	=\Si x\frac{d^2}{dx^2}+(\mu-\nu-(\nu+\mu+2)x)\frac d{dx}.
\end{displaymath}
It is obvoious that
\begin{displaymath}
	D_{x,\nu,\mu}=(1-x)^{-\nu}(1+x)^{-\mu}
		\frac d{dx}(1-x)^{\nu+1}(1+x)^{\mu+1}\frac d{dx}.
\end{displaymath}

We say that $g(x)\in\AD$ if $g(x)\in\Lp$,
the derivative $g'(x)$ is absolutely continuous
on every segment $[a,b]\subset(-1,1)$,
and $\Dx g(x)\in\Lp$.

Let
\begin{displaymath}
	\K=\inf_{g\in\AD}\prn{\norm{f-g}+\delta^2\norm{\Dx g(x)}}
\end{displaymath}
denote the~$K$-functional of Peetre interpolating
between spaces~$\Lp$ and~$\AD$.

We define the operator of generalised translation~$\hatT t{f,x}$
by
\begin{multline*}
	\hatT t{f,x}=\frac1{\pi\Si x\Co t}\\
		\times\int_0^\pi\bigg(
			2\prn{\sqrt{1-x^2}\cos t+x\sin t\cos\varphi
				+\sqrt{1-x^2}(1-\cos t)\sin^2\varphi}^2\\
			-1+\prn{x\cos t-\sqrt{1-x^2}\sin t\cos\varphi}^2
		\bigg)
			f\prn{x\cos t-\sqrt{1-x^2}\sin t\cos\varphi}\,d\varphi.
\end{multline*}

By means of the operator of generalised translation,
for a function $f\in\Lp$,
we define the generalised modulus of smoothness as follows
\begin{displaymath}
	\w=\sup_{|t|\le\delta}\norm{\hatT t{f,x}-f(x)}.
\end{displaymath}

Put $y=\cos t$, $z=\cos\varphi$ in the operator $\T t{f,x}$,
we denote it by~$\T y{f,x}$ and rewrite it in the form
\begin{displaymath}
	\T y{f,x}=\frac4{\pi\Si x(1+y)^2}
		\int_{-1}^1B_y(x,z,R)f(R)\frac{dz}{\sqrt{1-z^2}},
\end{displaymath}
where

\begin{gather*}
	R=xy-z\sqrt{1-x^2}\sqrt{1-y^2},\\
	B_y(x,z,R)
		=2\prn{\sqrt{1-x^2}y+zx\sqrt{1-y^2}
		  +\sqrt{1-x^2}(1-y)\Si z}^2-\Si R.
\end{gather*}

By~$P_\nu^{(\alpha,\beta)}(x)$ $(\nu=0,1,\dotsc)$
we denote the Jacobi polynomials, i.e.,
the algebraic polynomials of degree~$\nu$,
orthogonal with the weight function
$(1-x)^{\alpha}(1+x)^{\beta}$
on the segment~$[-1,1]$,
and normed by the condition
\begin{displaymath}
	P_\nu^{(\alpha,\beta)}(1)=1 \quad(\nu=0,1,\dotsc).
\end{displaymath}

Denote by~$a_n(f)$ the Fourier--Jacobi coefficients
of a function~$f$,
integrable with the weight function $\Si{x}^2$
on the segment~$[-1,1]$,
with respect to the system of Jacobi polynomials
$\brc{\Px n(x)}_{n=0}^\infty$,
i.e., let
\begin{displaymath}
	a_n(f)=\int_{-1}^1f(x)\Px n(x)\Si{x}^2\,dx\quad(n=0,1,\dotsc).
\end{displaymath}

The following symmetric operator of generalised translation
will play an auxiliary role in the sequel:
\begin{displaymath}
	\sT=\frac8{3\pi}\int_{-1}^1\Si{z}^2f(R)\frac{dz}{\sqrt{1-z^2}},
\end{displaymath}
where
\begin{displaymath}
	R=xy-z\sqrt{1-x^2}\sqrt{1-y^2}.
\end{displaymath}

\subsection{}

\begin{lmm}\label{lm:properties-tau}
	The operator~$\T y{f,x}$ has the following properties:
	\begin{enumerate}
	\item\label{it:properties-tau-1}
		it is linear,
	\item\label{it:properties-tau-2}
		$\T1{f,x}=f(x)$,
	\item\label{it:properties-tau-3}
		$\T y{\Px\nu,x}=\Px\nu(x)\Py\nu(y)
			\quad(\nu=0,1,\dotsc)$,
	\item\label{it:properties-tau-4}
		$\T y{1,x}=1$,
	\item\label{it:properties-tau-6}
		$a_n(\T y{f,x})=a_n(f)\Py n(y)
			\qquad(n=0,1,\dotsc)$.
	\end{enumerate}
\end{lmm}

\begin{proof}
	Properties~\ref{it:properties-tau-1} and~\ref{it:properties-tau-2}
	follow immediately from the definition
	of the operator~$\T y{f,x}$.
	
	In order to prove~\ref{it:properties-tau-3},
	we consider the functions
	\begin{multline*}
		P_{mn}^l(z)=\\
		P_\nu^{(\alpha,\beta)}(z)\binom{\nu+\alpha}\alpha2^{-m}i^{m-n}
			\sqrt{\frac{(l-m)!(l+m)!}{(l-n)!(l+n)!}}
			(1-z)^{(m-n)/2}(1+z)^{(m+n)/2},
	\end{multline*}
	where
	\begin{displaymath}
		l=\nu+\frac{\alpha+\beta}2,
		\quad
		m=\frac{\alpha+\beta}2,
		\quad
		n=\frac{\alpha-\beta}2.
	\end{displaymath}
	Putting $n=0$, $m=k=2$ in the formula of multiplication
	for functions~$P_{mn}^l$
	(see~\cite[p.~138]{vilenkin:spetsiyalnye},
	we obtain the required equalities.
	
	Property~\ref{it:properties-tau-4}
	is proved by means of $\Px0(x)$
	in~\ref{it:properties-tau-3}.
	
	We prove the equality in~\ref{it:properties-tau-6}.
	To this effect,
	consider
	\begin{multline*}
		I=a_n(\T y{f,x})=\int_{-1}^1\T y{f,x}\Px n(x)\Si{x}^2\,dx=\\
		=\frac4{\pi(1+y)^2}\int_{-1}^1\int_{-1}^1B_y(x,z,R)f(R)\Px n(x)\Si x
			\frac{dz\,dx}{\sqrt{1-z^2}},
	\end{multline*}
	where
	\begin{gather*}
		R=xy-z\sqrt{1-x^2}\sqrt{1-y^2},\\
		B_y(x,z,R)
			=2\prn{\sqrt{1-x^2}y+zx\sqrt{1-y^2}
			  +\sqrt{1-x^2}(1-y)\Si z}^2-\Si R.
	\end{gather*}
	Performing the change of variables
	\begin{gather}\label{eq:x-z}
		x=Ry+V\sqrt{1-R^2}\sqrt{1-y^2},\\
		z=-\frac{R\sqrt{1-y^2}-Vy\sqrt{1-R^2}}
			{\sqrt{1-\prn{Ry+V\sqrt{1-R^2}\sqrt{1-y^2}}^2}}\notag
	\end{gather}
	in the double integral,
	we obtain
	\begin{displaymath}
		I=\frac4{\pi(1+y)^2}
		  \int_{-1}^1\int_{-1}^1B_y(R,V,x)\Si Rf(R)\Px n(x)
			  \frac{dV\,dR}{\sqrt{1-V^2}}.
	\end{displaymath}
	Therefore,
	\begin{multline*}
		I=\int_{-1}^1f(R)\Si{R}^2\frac4{\pi\Si R(1+y)^2}
				\int_{-1}^1B_y(R,V,x)\Px n(x)\\
			\times\frac{dV}{\sqrt{1-V^2}}dR
				=\int_{-1}^1f(R)\Si{R}^2\T y{\Px n,R}\,dR.
	\end{multline*}
	Hence property~\ref{it:properties-tau-3} yields
	\begin{displaymath}
		I=\Py n(y)\int_{-1}^1f(R)\Px n(R)\Si{R}^2\,dR
		=a_n(f)\Py n(y).
	\end{displaymath}
	
	Lema~\ref{lm:properties-tau} is proved.
\end{proof}

\begin{lmm}\label{lm:inequality}
	Let the numbers~$p$ and~$\alpha$ be such that $\allp$;
	\begin{alignat*}2
		1/2 			   &<\alpha\le1
		  &\quad &\text{for $p=1$},\\
		1-\frac1{2p} &<\alpha<\frac32-\frac1{2p}
		  &\quad &\text{for $1<p<\infty$},\\
		1 				   &\le\alpha<3/2
		  &\quad &\text{for $p=\infty$}.
	\end{alignat*}
	\begin{displaymath}
		R=x\cos t-z\sqrt{1-x^2}\sin t.
	\end{displaymath}
	Then for every function~$f\in\Lp$,
	we have
	\begin{displaymath}
		\norm{\frac1{1-x^2}\int_{-1}^1\Si{R}|f(R)|\frac{dz}{\sqrt{1-z^2}}}
		\le C\norm{f},
	\end{displaymath}
	where the constant~$C$ does not depend on~$f$ and~$x$.
\end{lmm}

Lemma~\ref{lm:inequality} is proved in a more generalised form
in~\cite{p-berisha:east-98}.

\begin{lmm}\label{lm:bound-tau}
	Let the numbers~$p$ and~$\alpha$ be such that $\allp$;
	\begin{alignat*}2
		1/2				&<\alpha\le1
		  &\quad &\text{for $p=1$},\\
		1-\frac1{2p}	&<\alpha<\frac32-\frac1{2p}
		  &\quad &\text{for $1<p<\infty$},\\
		1				&\le\alpha<3/2
		  &\quad &\text{for $p=\infty$}.
	\end{alignat*}
	If $f\in\Lp$,
	then
	\begin{displaymath}
		\norm{\hatT t{f,x}}\le\frac C{\Co t}\norm f,
	\end{displaymath}
	where constant~$C$ does not depend on~$f$ and~$t$.
\end{lmm}

\begin{proof}
	Let
	\begin{displaymath}
		I=\norm{\hatT t{f,x}}
			=\frac1{\pi\Co t}\norm{\frac1{1-x^2}\int_{-1}^1B_{\cos t}(x,z,R)
				f(R)\frac{dz}{\sqrt{1-z^2}}},
	\end{displaymath}
	where
	\begin{gather*}
		R=x\cos t-z\sqrt{1-x^2}\sin t,\\
		B_y(x,z,R)
			=2\prn{\sqrt{1+x^2}y+zx\sqrt{1-y^2}
			  +\sqrt{1-x^2}(1-y)\Si z}^2-\Si R.
	\end{gather*}
	
	Since
	\begin{displaymath}
		R^2+\prn{\sqrt{1-x^2}y+zx\sqrt{1-y^2}}^2=1-\Si y\Si z,
	\end{displaymath}
	we have
	\begin{equation}\label{eq:partB}
		\left|\sqrt{1-x^2}y+zx\sqrt{1-y^2}\right|\le\sqrt{1-R^2}
	\end{equation}
	and
	\begin{displaymath}
		\Si y\Si z\le1-R^2.
	\end{displaymath}
	Since~$R$ is symmetric in~$x$ and~$y$,
	the last inequality yields
	\begin{displaymath}
		\Si x\Si z\le1-R^2.
	\end{displaymath}
	Applying this inequality and inequality~\eqref{eq:partB},
	we get
	\begin{displaymath}
		|B_y(x,z,R)|\le19\Si R.
	\end{displaymath}
	
	Applying Lemma~\ref{lm:inequality},
	we obtain
	\begin{displaymath}
		I\le\frac{\Cn}{\Co t}
			\norm{\frac1{1-x^2}\int_{-1}^1\Si{R}|f(R)|
			  \frac{dz}{\sqrt{1-z^2}}}
		\le\frac{\Cn}{\Co t}\norm f.
	\end{displaymath}
	
	Lemma~\ref{lm:bound-tau} is proved.
\end{proof}

\begin{lmm}\label{lm:properties-tau-5}
	If $g(x)\T y{f,x}\in\Lmu$ for each $y\in(-1,1)$, then
	\begin{displaymath}
		\int_{-1}^1f(x)\T y{g,x}\Si{x}^2\,dx
		=\int_{-1}^1g(x)\T y{f,x}\Si{x}^2\,dx.
	\end{displaymath}
\end{lmm}

\begin{proof}
	We have
	\begin{multline*}
		I=\int_{-1}^1f(x)\T y{g,x}\Si{x}^2\,dx\\
		=\frac4{\pi(1+y)^2}\int_{-1}^1\int_{-1}^1f(x)g(R)B_y(x,z,R)
			\Si x\frac{dz\,dx}{\sqrt{1-z^2}},
	\end{multline*}
	where
	\begin{gather*}
		R=x\cos t-z\sqrt{1-x^2}\sin t,\\
		B_y(x,z,R)
			=2\prn{\sqrt{1-x^2}y+zx\sqrt{1-y^2}
			  +\sqrt{1-x^2}(1-y)\Si z}^2-\Si R.
	\end{gather*}
	Performing the change of variables in this double integral
	by formulas~\eqref{eq:x-z},
	we obtain
	\begin{multline*}
		I=\frac4{\pi(1+y)^2}\int_{-1}^1\int_{-1}^1f(x)g(R)
			B_y(R,V,x)\Si R\frac{dV\,dR}{\sqrt{1-V^2}}\\
		=\int_{-1}^1g(R)\T y{f,R}\Si{R}^2\,dR.
	\end{multline*}
	
	Lemma~\ref{lm:properties-tau-5} is proved.
\end{proof}

\begin{lmm}\label{lm:Dtau}
	Assume that the derivative~$f'(x)$ is absolutely continuous
	on every segment $[a,b]\subset(-1,1)$
	and $\Dx f(x)\in\Lmu$.
	Then
	\begin{enumerate}
	\item\label{it:Dtau-1}
		for fixed~$y\in(-1,1)$,
		the derivative
		$\frac d{dx}\T y{f,x}$ is absolutely continuous
		on every segment $[c,d]\subset(-1,1)$,
	\item\label{it:Dtau-3}
		for almost every $x\in(-1,1)$
		and every $y\in(-1,1)$,
		the following equality holds true
		\begin{displaymath}
			\T y{\Dx f,x}=\Dx\T y{f,x}.
		\end{displaymath}
	\end{enumerate}
\end{lmm}

\begin{proof}
	In order to prove~\ref{it:Dtau-1},
	we consider the function
	\begin{displaymath}
		\varphi(x)=\frac{B_y(x,z,R)}{\Si x(1+y)^2\sqrt{1-z^2}}f(R),
	\end{displaymath}
	where $B_y(x,z,R)$ and~$R$
	have been defined in Lemma~\ref{lm:properties-tau-5}.
	It is obvious that the function $\varphi'(x)$ is continuous
	on every segment $[c,d]\subset(-1,1)$.
	Hence~\ref{it:Dtau-1} follows
	by applying Lebesgue's dominated convergence theorem.
	
	In order to prove~\ref{it:Dtau-3},
	first we prove the equality
	\begin{equation}\label{eq:Dxtau}
		\T y{\Dx f,x}=\Dx\T y{f,x}
	\end{equation}
	for infinitely differentiable functions~$f$
	which are equal to zero outside of some segment
	$[a,b]\subset(-1,-y)\cup(-y,y)\cup(y,1)$.
	
	From~\ref{it:Dtau-1} it follows that~$\Dx\T y{f,x}$ exists.
	
	Assume that the function~$f$ is infinitely differentiable
	and is equal to zero outside of some segment
	$[a,b]\subset(-1,-y)\cup(-y,y)\cup(y,1)$.
	Applying Lemmas~\ref{lm:properties-tau-5}
	and~\ref{lm:properties-tau},
	we obtain
	\begin{multline*}
		I=\int_{-1}^1\T y{\Dx f,x}\Px n(x)\Si{x}^2\,dx\\
		=\Py n(y)\int_{-1}^1\Dx f(x)\Px n(x)\Si{x}^2\,dx.
	\end{multline*}
	
	Integrating by parts twice
	and taking into account that $f(x)=0$
	and $f'(x)=0$
	outside of $[a,b]\subset(-1,1)$,
	we have
	\begin{displaymath}
		I=\Py n(y)\int_{-1}^1\Dx\Px n(x)f(x)\Si{x}^2\,dx.
	\end{displaymath}
	It is well known~\cite[p.~171]{erdelyi-m-o-t:transcendental}
	that
	\begin{displaymath}
		\Dx\Px n(x)=-n(n+5)\Px n(x).
	\end{displaymath}
	Therefore
	\begin{displaymath}
		I=-n(n+5)\Py n(y)\int_{-1}^1f(x)\Px n(x)\Si{x}^2\,dx.
	\end{displaymath}
	Applying Lemmas~\ref{lm:properties-tau}
	and~\ref{lm:properties-tau-5},
	integrating by parts twice,
	and considering that $\T y{f,x}=0$
	outside of some segment $[\gamma,\delta]\subset(-1,1)$,
	we obtain
	\begin{displaymath}
		I=\int_{-1}^1\Dx\T y{f,x}\Px n(x)\Si{x}^2\,dx.
	\end{displaymath}
	Thus for fixed~$y$,
	all the Fourier--Jacobi coefficients of the function
	\begin{displaymath}
		F(x)=\T y{\Dx f,x}-\Dx\T y{f,x}
	\end{displaymath}
	with respect to the system $\brc{\Px n(x)}_{n=0}^\infty$
	of polynomials are equal to zero.
	Hence it follows that $F(x)=0$
	almost everywhere on~$[-1,1]$.
	
	Thus,
	equality~\eqref{eq:Dxtau} has been proved
	for infinitely differentiable functions
	which are equal to zero outside of some segment
	$[a,b]\subset(-1,-y)\cup(-y,y)\cup(y,1)$.
	
	Now, let the function~$f(x)$ satisfy the conditions of the lemma.
	Let a function~$g(x)$ be infinitely differentiable
	and equal to zero outside of some segment
	$[c,d]\subset(-1,-y)\cup(-y,y)\cup(y,1)$.
	Integrating by parts twice and taking into account that
	\begin{gather*}
		g(x)\Si{x}^3\frac d{dx}\T y{f,x}\to0 \quad\text{and}
		\quad\T y{f,x}\Si{x}^3\frac d{dx}g(x)\to0\\
		\text{for} \quad x\to-1+0 \quad\text{and} \quad x\to1-0,
	\end{gather*}
	we obtain
	\begin{displaymath}
		J_1=\int_{-1}^1\Dx\T y{f,x}g(x)\Si{x}^2\,dx
		=\int_{-1}^1\Dx g(x)\T y{f,x}\Si{x}^2\,dx.
	\end{displaymath}
	Applying Lemma~\ref{lm:properties-tau},
	we get
	\begin{displaymath}
		J_1=\int_{-1}^1f(x)\T y{\Dx g,x}\Si{x}^2\,dx.
	\end{displaymath}
	On the other hand, let
	\begin{displaymath}
		J_2=\int_{-1}^1\T y{\Dx f,x}g(x)\Si{x}^2\,dx.
	\end{displaymath}
	Applying Lemma~\ref{lm:properties-tau-5}
	and then integrating by parts twice,
	we have
	\begin{displaymath}
		J_2=\int_{-1}^1\Dx\T y{g,x}f(x)\Si{x}^2\,dx.
	\end{displaymath}
	
	Therefore, we obtain
	\begin{displaymath}
		J_2-J_1
		=\int_{-1}^1\prn{\Dx\T y{g,x}-\T y{\Dx g,x}}f(x)\Si{x}^2\,dx.
	\end{displaymath}
	But for the function~$g(x)$
	we have already proved equality~\eqref{eq:Dxtau}
	for almost every $x\in(-1,1)$.
	Hence
	\begin{displaymath}
		J_2-J_1
		=\int_{-1}^1\prn{\T y{\Dx f,x}-\Dx\T y{f,x}}g(x)\Si{x}^2\,dx=0
	\end{displaymath}
	for every~$y$.
	Now,
	equality~\eqref{eq:Dxtau}
	follows from the fact that the segment
	$[c,d]\subset(-1,-y)\cup(-y,y)\cup(y,1)$
	and the function~$g(x)$ can be arbitrarily chosen.
	
	Lemma~\ref{lm:Dtau} is proved.
\end{proof}

\begin{lmm}\label{lm:tauuDx}
	Assume that the derivative~$f'(x)$ is absolutely continuous
	on every segment $[a,b]\subset(-1,1)$
	and $\Dx f(x)\in\Lmu$.
	Then for almost every $x\in(-1,1)$
	and every $y\in(-1,1)$
	\begin{equation}\label{eq:tauuDx}
		\T y{f,x}-f(x)=\int_1^y(1-v)^{-1}(1+v)^{-5}\int_1^v(1+u)^4
			\T u{\Dx f,x}\,du\,dv
	\end{equation}
	and
	\begin{multline}\label{eq:tauuDx0}
		\T y{f,x}-\T0{f,x}\\
		=-\int_0^y(1-v)^{-1}(1+v)^{-5}
		  \int_v^{-1}(1+u)^4\T u{\Dx f,x}\,du\,dv.
	\end{multline}
\end{lmm}

\begin{proof}
	We prove equality~\eqref{eq:tauuDx}.
	If~$f$ is an infinitely differentiable function,
	equal to zero outside of some segment
	$[a,b]\subset(-1,-y)\cup(-y,y)\cup(y,1)$,
	then for almost every $x\in(-1,1)$
	and almost every $u\in(-1,1)$
	the following equality holds true
	\begin{displaymath}
		\T u{\Dx f,x}=\Du\T u{f,x}.
	\end{displaymath}
	
	Applying this equality and Lemma~\ref{lm:properties-tau},
	we obtain
	\begin{multline*}
		\int_1^y(1-v)^{-1}(1+v)^{-5}\int_1^v(1+u)^4\T u{\Dx f,x}\,du\,dv\\
		=\int_1^y(1-v)^{-1}(1+v)^{-5}\int_1^v(1+u)^4\Du\T u{f,x}\,du\,dv
			=\T y{f,x}-f(x).
	\end{multline*}
	
	Now let the function~$f(x)$ satisfy the conditions of the lemma
	and let~$g(x)$ be an infinitely differentiable function,
	equal to zero outside of some segment
	$[c,d]\subset(-1,-y)\cup(-y,y)\cup(y,1)$.
	Then by Lemma~\ref{lm:properties-tau-5},
	analogously to the proof of Lemma~\ref{lm:Dtau},
	while integrating by parts twice,
	it is easy to prove that
	\begin{multline*}
		J=\int_{-1}^1\int_1^y(1-v)^{-1}(1+v)^{-5}
			\int_1^v(1+u)^4\T u{\Dx f,x}g(x)\Si{x}^2\,du\,dv\,dx\\
		=\int_{-1}^1f(x)\Si{x}^2\int_1^y(1-v)^{-1}(1+v)^{-5}
			\int_1^v(1+u)^4\Dx\T u{g,x}\,du\,dv\,dx.
	\end{multline*}
	Making use of Lemma~\ref{lm:Dtau}
	and the fact that we have already proved equality~\eqref{eq:tauuDx}
	for almost every $x\in(-1,1)$
	in the case of any infinitely differentiable function $g(x)$,
	equal to zero outside of the segment
	$[c,d]\subset(-1,-y)\cup(-y,y)\cup(y,1)$,
	we obtain
	\begin{displaymath}
		J=\int_{-1}^1(\T y{g,x}-g(x))f(x)\Si{x}^2\,dx.
	\end{displaymath}
	Applying once more Lemma~\ref{lm:properties-tau-5},
	we get that
	\begin{displaymath}
		J=\int_{-1}^1(\T y{f,x}-f(x))g(x)\Si{x}^2\,dx.
	\end{displaymath}
	Hence equality~\eqref{eq:tauuDx} follows
	by taking into account the fact that the segment~$[c,d]$
	and the function~$g(x)$
	can be arbitrarily chosen.
	
	Equality~\eqref{eq:tauuDx0} is proved in an analogous way.
	
	Lemma~\ref{lm:tauuDx} is proved.
\end{proof}

\begin{cor}
	Assume that the derivative~$f'(x)$ is absolutely continuous
	on every segment $[a,b]\subset(-1,1)$
	and $\Dx f(x)\in\Lmu$.
	Then for almost every $x\in(-1,1)$
	and every $t\in(-\pi,\pi)$
	\begin{multline*}
		\hatT t{f,x}-f(x)\\
		=\int_0^t\sincosv\int_0^v\hatT u{\Dx f,x}\sincosu\,du\,dv
	\end{multline*}
	and
	\begin{multline*}
		\hatT t{f,x}-\hatT{\pi/2}{f,x}\\
		=-\int_{\pi/2}^t\sincosv
		  \int_v^\pi\hatT u{\Dx f,x}\sincosu\,du\,dv.
	\end{multline*}
\end{cor}

The first equality follows immediately
from equality~\eqref{eq:tauuDx}
by substituting~$\cos u$ and~$\cos v$ for~$u$ and~$v$,
respectively.
In an analogous way,
the second equality follows from equality~\eqref{eq:tauuDx0}.

\begin{lmm}\label{lm:bernshtein-markov}
	Let~$P_n$ be an algebraic polynomial
	of degree not greater than $n-1$,
	$\allp$, $\rho\ge0$;
	\begin{alignat*}2
		\alpha &>-1/p &\quad &\text{for $1\le p<\infty$},\\
		\alpha &\ge0  &\quad &\text{for $p=\infty$}.
	\end{alignat*}
	Then the following inequalities hold true:
	\begin{gather*}
		\|P'_n\|_{p,\alpha+1/2}\le\Cn n\norm{P_n},\\
		\norm{P_n}\le\Cn n^{2\rho}\|P_n\|_{p,\alpha+\rho},
	\end{gather*}
	where the constants~$\prevC$ and~$\lastC$
	do not depend on~$n$.
\end{lmm}

Lemma is proved in~\cite{potapov:vestnik-60a}.

\begin{lmm}\label{lm:S-Q}
	Let~$q$ and~$m$ be natural numbers
	and let $f\in\Lmu$.
	Then the function
	\begin{displaymath}
		Q(x)=\int_0^\pi T_{2;\cos t}(f,x)\krn t\sin^5t\,dt
	\end{displaymath}
	is an algebraic polynomial
	of degree not greater than $(q+2)\*(m-1)$.
\end{lmm}

Lemma is proved in~\cite{potapov:trudy-74}.

\begin{lmm}\label{lm:E-D}
	Let the numbers~$p$ and~$\alpha$ be such that $\allp$;
	\begin{alignat*}2
		-1/2 			  &<\alpha\le2
		  &\quad &\text{for $p=1$},\\
		-\frac1{2p} &<\alpha<5/2-\frac1{2p}
		  &\quad &\text{for $1<p<\infty$},\\
		0 				  &\le\alpha<5/2
		  &\quad &\text{for $p=\infty$}.
	\end{alignat*}
	If $f\in\AD$,
	then
	\begin{displaymath}
		\E\le C\frac1{n^2}\norm{\Dx f(x)},
	\end{displaymath}
	where the constant~$C$ does not depend on~$f$ and~$n$.
\end{lmm}

\begin{proof}
	For a fixed natural number $q>2$,
	we chose the natural number~$m$
	such that
	\begin{displaymath}
		\frac{n-1}{q+2}<m\le\frac{n-1}{q+2}+1.
	\end{displaymath}
	
	It is easy to prove that under the conditions of the lemma,
	$f\in\Lp$ implies $f\in\Lmu$.
	Hence by Lemma~\ref{lm:S-Q},
	it follows that the function
	\begin{displaymath}
		Q(x)=\frac1{\gamma_m}
		  \int_0^\pi T_{2;\cos t}(f,x)\krn t\sin^5t\,dt,
	\end{displaymath}
	where
	\begin{displaymath}
		\gamma_m=\int_0^\pi\krn t\sin^5t\,dt,
	\end{displaymath}
	is an algebraic polynomial of degree not greater than $n-1$.
	Therefore,
	by applying the generalised Minkowski inequality,
	we have
	\begin{multline*}
		\E\le\norm{f-Q}\\
		\le\frac1{\gamma_m}\int_0^\pi\norm{T_{2;\cos t}(f,x)-f(x)}
			\krn t\sin^5t\,dt.
	\end{multline*}
	
	Reasoning as in the proof of inequality~\eqref{eq:partB}
	of Theorem~\ref{th:w-K}, i.e.,
	applying the appropriately modified versions
	of Lemmas~\ref{lm:tauuDx} and~\ref{lm:bound-tau}
	for the operator $T_{2;\cos t}(f,x)$,
	we obtain
	\begin{displaymath}
		\E\le\Cn\norm{\Dx f(x)}\frac1{\gamma_m}
		  \int_0^\pi t^2\krn t\sin^5t\,dt.
	\end{displaymath}
	Making use of an estimate of Jackson kernel,
	we get
	\begin{displaymath}
		\E\le\Cn\frac1{m^2}\norm{\Dx f(x)}
		\le\Cn\frac1{n^2}\norm{\Dx f(x)}.
	\end{displaymath}
	
	Lema~\ref{lm:E-D} is proved.
\end{proof}

\subsection{}

\begin{thm}\label{th:w-K}
	Let the numbers~$p$ and~$\alpha$ be such that $\allp$;
	\begin{alignat*}2
		1/2 			   &<\alpha\le1
		  &\quad &\text{for $p=1$},\\
		1-\frac1{2p} &<\alpha<\frac32-\frac1{2p}
		  &\quad &\text{for $1<p<\infty$},\\
		1 				   &\le\alpha<3/2
		  &\quad &\text{for $p=\infty$}.
	\end{alignat*}
	If $f\in\Lp$,
	then for all $\delta\in[0,\pi)$,
	\begin{displaymath}
		\Cn\K\le\w\le\Cn\frac1{\Co\delta}\K,
	\end{displaymath}
	where the positive constants~$\prevC$ and~$\lastC$
	do not depend on~$f$ and~$\delta$.
\end{thm}

\begin{proof}
	We prove that for every function $g(x)\in\AD$
	and every $t\in(-\pi,\pi)$,
	we have
	\begin{equation}\label{eq:Dl-Dx}
		\norm{\hatT t{g,x}-g(x)}
		\le\Cn\frac1{\Co t}t^2\norm{\Dx g(x)},
	\end{equation}
	where the constant~$\lastC$ does not depend on~$g$ dhe~$t$.
	
	Let $0<t\le\pi/2$.
	Then Corollary of Lemma~\ref{lm:tauuDx} yields
	\begin{multline*}
		I_1=\norm{\hatT t{g,x}-g(x)}\\
		=\norm{\int_0^t\sincosv\int_0^v\hatT u{\Dx g,x}\sincosu\,du\,dv}.
	\end{multline*}
	Applying the generalised Minkowski inequality
	and Lemma~\ref{lm:bound-tau},
	we get
	\begin{multline*}
		I_1\le\int_0^t\sincosv\\
			\times\int_0^v\norm{\hatT u{\Dx g,x}}\sincosu\,du\,dv\\
		\le\Cn\norm{\Dx g(x)}\int_0^t\sincosv\int_0^v\varsincosu\,du\,dv.
	\end{multline*}
	Since the inequality
	\begin{displaymath}
		\int_0^t\sincosv\int_0^v\varsincosu\,du\,dv\le\Cn t^2
	\end{displaymath}
	holds for $0<t\le\pi/2$,
	we obtain
	\begin{displaymath}
		I_1\le\Cn t^2\norm{\Dx g(x)}
		\le\lastC\frac1{\Co t}t^2\norm{\Dx g(x)}.
	\end{displaymath}
	
	For $t=0$ inequality~\eqref{eq:Dl-Dx} is trivial.
	
	Let $\pi/2\le t<\pi$.
	Then by Corollary of Lemma~\ref{lm:tauuDx},
	we get
	\begin{multline*}
		I_2=\norm{\hatT t{g,x}-\hatT{\pi/2}{g,x}}\\
		=\norm{\int_{\pi/2}^t\sincosv
		  \int_v^\pi\hatT u{\Dx g,x}\sincosu\,du\,dv}.
	\end{multline*}
	Applying the generalised Minkowski inequality
	and then Lemma~\ref{lm:bound-tau},
	we have
	\begin{multline*}
		I_2\le\Cn\norm{\Dx g(x)}\\
			\times\int_{\pi/2}^t\sincosv\int_v^\pi\varsincosu\,du\,dv.
	\end{multline*}
	Considering that for $\pi/2\le t<\pi$ we have
	\begin{displaymath}
		\int_{\pi/2}^t\sincosv\int_v^\pi\varsincosu\,du\,dv
		\le\Cn\frac1{\Co t},
	\end{displaymath}
	it follows that
	\begin{equation}\label{eq:I2-Dx}
		I_2\le\Cn\frac1{\Co t}\norm{\Dx g(x)}
		\le\lastC\frac1{\Co t}t^2\norm{\Dx g(x)}.
	\end{equation}
	Since
	\begin{displaymath}
		\norm{\hatT t{g,x}-g(x)}
		\le\norm{\hatT t{g,x}-\hatT{\pi/2}{g,x}}
			+\norm{\hatT{\pi/2}{g,x}-g(x)},
	\end{displaymath}
	applying inequality~\eqref{eq:I2-Dx}
	and the fact that inequality~\eqref{eq:Dl-Dx}
	has been proved for $0\le t\le\pi/2$,
	we obtain
	\begin{displaymath}
		\norm{\hatT t{g,x}-g(x)}\le\Cn\frac1{\Co t}t^2\norm{\Dx g(x)}
	\end{displaymath}
	for $\pi/2\le t<\pi$.
	
	Thus,
	inequality~\eqref{eq:Dl-Dx} is proved for $0\le t<\pi$.
	Since
	\begin{displaymath}
		\T{\cos t}{g,x}=\T{\cos{(-t)}}{g,x},
	\end{displaymath}
	we conclude that inequality~\eqref{eq:Dl-Dx}
	holds for every $t\in(-\pi,\pi)$.
	
	Let $f\in\Lp$ and $0\le|t|\le\delta<\pi$.
	Then for every function $g(x)\in\AD$,
	applying Lemma~\ref{lm:bound-tau} gives
	\begin{multline*}
		\norm{\hatT t{f,x}-f(x)}
		  \le\norm{\hatT t{f-g,x}}+\norm{\hatT t{g,x}-g(x)}+\norm{g-f}\\
		\le\Cn\frac1{\Co t}\norm{f-g}+\norm{\hatT t{g,x}-g(x)}.
	\end{multline*}
	Making use of inequality~\eqref{eq:Dl-Dx},
	we get
	\begin{displaymath}
		\norm{\hatT t{f,x}-f(x)}
		\le\Cn\frac1{\Co t}\prn{\norm{f-g}+t^2\norm{\Dx g(x)}},
	\end{displaymath}
	where the constant~$\lastC$ does not depend on~$f$, $g$ and~$t$.
	This proves the right-hand side inequality of the theorem.
	
	In order to prove the left-hand side inequality,
	we consider the function
	\begin{displaymath}
		\gd=\frac1\kap
		  \int_0^\delta\sincosv\int_0^v\hatT u{f,x}\sincosu\,du\,dv,
	\end{displaymath}
	where
	\begin{displaymath}
		\kap=\int_0^\delta\sincosv\int_0^v\sincosu\,du\,dv.
	\end{displaymath}
	
	Let $0<\delta<\pi/2$,
	then
	\begin{equation}\label{eq:kap-delta}
		\Cn\delta^2\le\kap\le\Cn\delta^2.
	\end{equation}
	
	Applying the generalised Minkowski inequality
	and Lemma~\ref{lm:bound-tau},
	we obtain
	\begin{multline*}
		\norm\gd\\
		\le\frac1{\kap}\int_0^\delta\sincosv
				\int_0^v\norm{\hatT u{f,x}}\sincosu\,du\,dv\\
		\le\Cn\frac1{\Co\delta}\norm f,
	\end{multline*}
	that is, $\gd\in\Lp$.
	
	Put
	\begin{displaymath}
		g(x)=-\int_0^x\Si{y}^{-3}\int_y^1\Si{z}^2\fz\,dz\,dy,
	\end{displaymath}
	where
	\begin{displaymath}
		c_1=\int_{-1}^1\Si{z}^2f(z)\,dz,
		\quad c_0=\int_{-1}^1\Si{z}^2\,dz.
	\end{displaymath}
	Since
	\begin{displaymath}
		\Dx g(x)=f(x)-\frac{c_1}{c_0},
	\end{displaymath}
	we have
	\begin{multline*}
		\gd=\frac1\kap\int_0^\delta\sincosv\\
			\times\int_0^v\hatT u{\Dx g,x}\sincosu\,du\,dv+\frac{c_1}{c_0}.
	\end{multline*}
		Applying Corollary of Lemma~\ref{lm:tauuDx} gives
	\begin{displaymath}
		\gd=\frac1\kap(\hatT\delta{g,x}-g(x))+\frac{c_1}{c_0}.
	\end{displaymath}
	Applying the operator~$\Dx$
	and then Lemma~\ref{lm:Dtau},
	it follows that
	\begin{displaymath}
		\Dx\gd=\frac1\kap(\hatT\delta{\Dx g,x}-\Dx g(x))
		=\frac1\kap(\hatT\delta{f,x}-f(x)).
	\end{displaymath}
	Therefore,
	by Lemmas~\ref{lm:bound-tau} and~\ref{lm:Dtau},
	we conclude that $\gd\in\AD$.
	
	By the last equality and inequality~\eqref{eq:kap-delta},
	we obtain
	\begin{displaymath}
		\norm{\Dx\gd}\le\Cn\frac1{\delta^2}\norm{\hatT\delta{f,x}-f(x)},
	\end{displaymath}
	that is,
	\begin{displaymath}
		\norm{\Dx\gd}\le\lastC\frac1{\delta^2}\w.
	\end{displaymath}
	
	On the other hand, by applying the Minkowsky inequality,
	we get
	\begin{multline*}
		\norm{f(x)-\gd}\le\frac1\kap\int_0^\delta\sincosv\\
			\times\int_0^v\norm{f(x)-\hatT u{f,x}}\sincosu\,du\,dv\le\w.
	\end{multline*}
	Thus, for $0<\delta\le\pi/2$ we have proved that
	\begin{displaymath}
		I(\delta)=\norm{f(x)-\gd}+\delta^2\norm{\Dx\gd}\le\Cn\w.
	\end{displaymath}
	
	Since for $\pi/2\le\delta<\pi$ we have
	$\delta^2<\pi^2\cdot1$ and $1<\pi/2$,
	it follows that
	\begin{multline*}
		\K\le\pi^2\prn{\norm{f(x)-g_1(x)}+1^2\cdot\norm{\Dx g_1(x)}}\\
		=\pi^2I(1)\le\pi^2\lastC\wpar{f,1}\le\Cn\w.
	\end{multline*}
	Thus,
	we have proved the left-hand side inequality of the theorem
	for $0<\delta<\pi$.
	For $\delta=0$ this inequality is trivial.
	
	Theorem~\ref{th:w-K} is proved.
\end{proof}

\begin{thm}\label{th:jackson}
	Let the numbers~$p$ and~$\alpha$ be such that $\allp$;
	\begin{alignat*}2
		1/2 			   &<\alpha\le1
		  &\quad &\text{for $p=1$},\\
		1-\frac1{2p} &<\alpha<\frac32-\frac1{2p}
		  &\quad &\text{for $1<p<\infty$},\\
		1 				   &\le\alpha<3/2
		  &\quad &\text{for $p=\infty$}.
	\end{alignat*}
	If $f\in\Lp$,
	then for every natural number~$n$
	\begin{displaymath}
		\Cn\E\le\wpar{f,1/n}
		\le\Cn\frac1{n^2}\sum_{\nu=1}^n\nu\Epar\nu f,
	\end{displaymath}
	where the positive constants~$\prevC$ and~$\lastC$
	do not depend on~$f$ and~$n$.
\end{thm}

\begin{proof}
	For every function $g(x)\in\AD$,
	we have
	\begin{displaymath}
		\E\le\Epar n{f-g}+\Epar n g.
	\end{displaymath}
	Applying Lemma~\ref{lm:E-D} gives
	\begin{displaymath}
		\E\le\norm{f-g}+\Cn\frac1{n^2}\norm{\Dx g(x)},
	\end{displaymath}
	where the constant~$\lastC$ does not depend on~$f$, $g$ and~$n$.
	Therefore, we get
	\begin{displaymath}
		\E\le\Cn\Kpar{f,1/n}.
	\end{displaymath}
	Hence Theorem~\ref{th:w-K} yields
	\begin{displaymath}
		\E\le\Cn\wpar{f,1/n},
	\end{displaymath}
	which proves the left-hand side inequality of the theorem.
	
	We prove the right-hand side inequality.
	Let $P_n(x)$ be the algebraic polynomial of best approximation for~$f$
	in the metrics~$\Lp$
	whose degree is not greater than~$n-1$.
	Let~$k$ be chosen such that
	\begin{equation}\label{eq:k}
		n/2<2^k\le n+1.
	\end{equation}
	Since $P_{2^k}(x)\in\AD$,
	Theorem~\ref{th:w-K} yields
	\begin{multline*}
		\wpar{f,1/n}\le\Cn\frac1{\prn{\cos\frac1{2n}}^4}\Kpar{f,1/n}\\
		\le\Cn\prn{\norm{f-P_{2^k}}+\frac1{n^2}\norm{\Dx P_{2^k}}}.
	\end{multline*}
	Since
	\begin{displaymath}
		\Dx P_{2^k}(x)
		=\sum_{\nu=0}^{k-1}\Dx\prn{P_{2^{\nu+1}}(x)-P_{2^\nu}(x)},
	\end{displaymath}
	Lemma~\ref{lm:bernshtein-markov} yields
	\begin{multline*}
		\norm{\Dx P_n(x)}\le\norm{\Si xP''_n(x)}+6\norm{P'_n(x)}\\
		\le\Cn n\|P'_n(x)\|_{p,\alpha+1/2}\le\Cn n^2\norm{P_n},
	\end{multline*}
	whence we obtain
	\begin{multline*}
		\norm{\Dx P_{2^k}(x)}
		\le\Cn\sum_{\nu=0}^{k-1}2^{2(\nu+1)}
			\norm{P_{2^{\nu+1}}(x)-P_{2^\nu}(x)}\\
		\le\lastC\sum_{\nu=0}^{k-1}2^{2(\nu+1)}
			(\norm{P_{2^{\nu+1}}(x)-f(x)}+\norm{f(x)-P_{2^\nu}(x)})\\
		\le\lastC\sum_{\nu=0}^{k-1}2^{2(\nu+1)}
			(\Epar{2^{\nu+1}}f+\Epar{2^\nu}f)
		\le2\lastC\sum_{\nu=0}^{k-1}2^{2(\nu+1)}\Epar{2^\nu}f.
	\end{multline*}
	Therefore, inequality~\eqref{eq:k} implies
	\begin{multline*}
		\wpar{f,1/n}
		\le\Cn\bigg(
			\Epar{2^k}f
			+\frac1{n^2}\sum_{\nu=0}^{k-1}2^{2(\nu+1)}\Epar{2^\nu}f
		\bigg)\\
		\le\lastC\frac1{n^2}\sum_{\nu=0}^k 2^{2(\nu+1)}\Epar{2^\nu}f.
	\end{multline*}
	
	We note that
	\begin{displaymath}
		\sum_{\mu=2^{\nu-1}}^{2^\nu-1}\mu\Epar\mu f
		\ge2^{2(\nu-1)}\Epar{2^\nu}f
	\end{displaymath}
	holds for $\nu=1,2,\dots,k$.
	Hence we have
	\begin{multline*}
		\wpar{f,1/n}\le\Cn\frac1{n^2}\bigg(4\Epar1f+\sum_{\nu=1}^k
			\sum_{\mu=2^{\nu-1}}^{2^\nu-1}\mu\Epar\mu f\bigg)\\
		\le\Cn\frac1{n^2}\sum_{\nu=1}^{2^k-1}\nu\Epar\nu f
		\le\lastC\frac1{n^2}\sum_{\nu=1}^n\nu\Epar\nu f.
	\end{multline*}
	
	Theorem~\ref{th:jackson} is proved.
\end{proof}

\bibliographystyle{amsplain}
\bibliography{maths}

\providecommand\cprime{'}
\providecommand{\bysame}{\leavevmode\hbox to3em{\hrulefill}\thinspace}
\providecommand{\MR}{\relax\ifhmode\unskip\space\fi MR }
\providecommand{\MRhref}[2]{%
  \href{http://www.ams.org/mathscinet-getitem?mr=#1}{#2}
}
\providecommand{\href}[2]{#2}
\begin{thebibliography}{1}

\bibitem{butzer-s-w:c-80}
P.~L. Butzer, R.~L. Stens, and M.~Wehrens, \emph{Higher order moduli of
  continuity based on the {J}acobi translation operator and best
  approximation}, C. R. Math. Rep. Acad. Sci. Canada \textbf{2} (1980), no.~2,
  83--88. \MR{81a:41040}

\bibitem{erdelyi-m-o-t:transcendental}
A.~Erd{\'e}lyi, W.~Magnus, F.~Oberhettinger, and F.~G. Tricomi, \emph{Higher
  transcendental functions}, Three volumes, Robert E. Krieger Publishing Co.
  Inc., Melbourne, Fla., 1981, (Russian translation, Gosudarstv. Izdat.
  Inostranno\u{\i} Literatury, Moscow, 1969). \MR{84h:33001}

\bibitem{pawelke:acta-72}
S.~Pawelke, \emph{Ein {S}atz vom {J}acksonschen {T}yp f\"ur algebraische
  {P}olynome}, Acta Sci. Math. (Szeged) \textbf{33} (1972), no.~3--4, 323--336.
  \MR{58 \#1848}

\bibitem{potapov:vestnik-60a}
M.~K. Potapov, \emph{Nekotorye neravenstva dlya polinomov i ikh proizvodnykh},
  Vestnik Moskov. Univ. Ser. I Mat. Mekh. (1960), no.~2, 10--20. \MR{23
  \#A3412}

\bibitem{potapov:trudy-74}
\bysame, \emph{O strukturnykh i konstruktivnykh kharakteristikakh nekotorykh
  klassov funktsi\u{\i}}, Trudy Mat. Inst. Steklov. \textbf{131} (1974),
  211--231, 247--248. \MR{50 \#14195}

\bibitem{potapov:vestnik-83}
\bysame, \emph{O priblizhenii algebraicheskimi mnogochlenami v
  in\-te\-g\-ra\-l\-'no\u{\i} me\-t\-ri\-ke s vesom {Y}akobi}, Vestnik Moskov.
  Univ. Ser. I Mat. Mekh. (1983), no.~4, 43--52. \MR{84i:41008}

\bibitem{p-berisha:east-98}
M.~K. Potapov and F.~M. Berisha, \emph{Approximation of classes of functions
  defined by a generalized $k$-th modulus of smoothness}, East J. Approx.
  \textbf{4} (1998), no.~2, 217--241. \MR{1638345 (99g:41007)}

\bibitem{potapov-f:trudy-85}
M.~K. Potapov and V.~M. Fedorov, \emph{O teoremakh {D}zheksona dlya
  obobshchennogo modulya gladkosti}, Trudy Mat. Inst. Steklov. \textbf{172}
  (1985), 291--298, 355. \MR{86m:41012}

\bibitem{vilenkin:spetsiyalnye}
N.~Ya. Vilenkin, \emph{Spetsiyal'nye funktsii i teoriya predstavlenii grupp},
  "Nauka", Moscow, 1991. \MR{93d:33013}

\end{thebibliography}

\end{document}